\documentclass[reqno]{amsart}
\usepackage{float}
\usepackage{amsmath}
\usepackage{graphicx}
\usepackage{latexsym}
\usepackage{amsfonts}
\usepackage{xcolor}
\usepackage{amssymb}
\theoremstyle{plain}
\newtheorem{theorem}{Theorem}
\newtheorem{corollary}[theorem]{Corollary}
\newtheorem{lemma}[theorem]{Lemma}

\theoremstyle{definition}

\newtheorem{example}[theorem]{Example}

\newtheorem*{remark*}{Remark}

\newcommand{\rd}{\mathbb R^d}
\newcommand{\R}{\mathbb R}
\newcommand{\pr}{\mathbf P}
\newcommand{\e}{\mathbf E}

\newcommand{\ra}{\rightarrow}

\begin{document}

\title[Green function of a random walk in a cone.]
{Green function of a random walk in a cone.}

\author[Duraj]{Jetlir Duraj}
\address{Department of Economics, Harvard University, USA }
\email{duraj@g.harvard.edu}

\author[Wachtel]{Vitali Wachtel}
\address{Institut f\"ur Mathematik, Universit\"at Augsburg, 86135 Augsburg, Germany}
\email{vitali.wachtel@math.uni-augsburg.de}

\begin{abstract}
This paper studies the asymptotic behavior of the Green function of a multidimensional random walk
killed when leaving a convex cone with smooth boundary. Our results imply uniqueness,
up to a multiplicative factor, of the positive harmonic function for the killed
random walk.
\end{abstract}

\maketitle
\section{Introduction, main results and discussion}

Consider a random walk $\{S(n),n\geq1\}$ on $\rd$, $d\geq1$, where
$$
S(n)=X(1)+\cdots+X(n)
$$
and $\{X(n), n\geq1\}$ is a family of independent copies of a random variable $X=(X_1,X_2,\ldots,X_d)$.
Denote by $\mathbb{S}^{d-1}$ the unit sphere of $\rd$ and $\Sigma$ an open and connected subset of
$\mathbb{S}^{d-1}$. Let $K$ be the cone generated by the rays emanating from the origin and passing
through $\Sigma$, i.e. $\Sigma=K\cap \mathbb{S}^{d-1}$.

Let $\tau_x$ be the exit time from $K$ of the random walk with starting point $x\in K$, that is,
$$
\tau_x=\inf\{n\ge 1: x+S(n)\notin K\}.
$$
Denisov and Wachtel \cite{DW15,DW18+} have constructed a positive harmonic function $V(x)$ for the random
walk $\{S(n)\}$ killed at leaving $K$. That is,
$$
V(x)=\mathbf{E}[V(x+X);\tau_x>1],\quad x\in K.
$$
They have also proved standard and local limit theorems for random walks conditioned to stay in the cone
$K$.

In the present paper we determine the asymptotic behavior of the Green function for
$\{S(n)\}$ killed at leaving $K$ and prove by using the Martin compactification the uniqueness of
the positive harmonic function for such processes.

We next introduce the assumptions on the cone $K$ and on the random walk $\{S(n):n\geq 1\}$. Let $u(x)$ be the
unique strictly positive  on $K$ solution of the following boundary problem:
$$
\Delta u(x)=0,\ x \in K\quad\text{with boundary condition }u\big|_{\partial K}=0.
$$
Let $L_{\mathbb{S}^{d-1}}$ be the Laplace-Beltrami operator on
$\mathbb{S}^{d-1}$ and assume that $\Sigma$ is regular with respect to $L_{\mathbb{S}^{d-1}}$. With this
assumption, there exists 
(see, for example, \cite{BS97})
a complete set of orthonormal eigenfunctions $m_j$  and corresponding
eigenvalues $0<\lambda_1<\lambda_2\le\lambda_3\le\ldots$ satisfying
\begin{align}
 \label{eq.eigen}
  L_{\mathbb{S}^{d-1}}m_j(x)&=-\lambda_jm_j(x),\quad x\in \Sigma\\
  \nonumber m_j(x)&=0, \quad x\in \partial \Sigma .
\end{align}
Define
$$
p:=\sqrt{\lambda_1+(d/2-1)^2}-(d/2-1)>0.
$$
The function $u(x)$ is given by
\begin{equation}
\label{u.from.m}
u(x)=|x|^pm_1\left(\frac{x}{|x|}\right),\quad x\in K.
\end{equation}
As in \cite{DW18+}, we shall impose the following condition on the cone $K$:
\begin{itemize}
\item $K$ is either convex or starlike and $C^2$.
\end{itemize}
We impose the following assumptions on the increments of the random walk:
\begin{itemize}
\item {\it Normalisation assumption:} We assume that $\mathbf EX_j=0,
  \mathbf EX_j^2=1,j=1,\ldots,d$. In addition we assume that
  $cov(X_i,X_j)=0$.
\item {\it Moment assumption:} We assume that $\mathbf
  E|X|^{\alpha}<\infty$ with $\alpha=p$ if $p>2$ and some $\alpha>2$
  if $p\le 2$.
\item {\it Lattice assumption:} $X$ takes values on a lattice $R$ which is
a non-degenerate linear transformation of $\mathbb{Z}^d$.  
\end{itemize}

Our first result describes the asymptotic behavior of the Green function
for endpoints $y$ which lie deep inside the cone $K$.
\begin{theorem}
\label{thm:deep}
Set $r_1(p)=p+d-2+(2-p)^+$ and assume that $\mathbf{E}|X|^{r_1(p)}$ is finite.
If $|y|\to\infty$ and $dist(y,\partial K)\ge \alpha|y|$ for some positive
$\alpha$, then
\begin{equation}
\label{T1}
G_K(x,y):=\sum_{n=0}^\infty\mathbf{P}(x+S(n)=y,\tau_x>n)
\sim cV(x)\frac{u(y)}{|y|^{2p+d-2}}.
\end{equation}
Moreover, this relation remains valid if one replaces the moment condition\\
$\mathbf{E}|X|^{r_1(p)}<\infty$ by the following restriction on the local
structure of $X_1$:
\begin{equation}
\label{local_assump}
\mathbf{P}(X=x)\le |x|^{-p-d+1}f(|x|)
\end{equation}
for some decreasing function $f$ such that $u^{(3-p)\vee 1}f(u)\to 0$ as
$u\to\infty$.
\end{theorem}

Uchiyama~\cite{Uchiyama98} has shown, see Theorem 2 there, that if $d\ge 5$ and
$\mathbf{E}|X|^{d-2}<\infty$ then
\begin{equation*}
G_{\mathbb{R}^d}(0,z)\sim\frac{c}{|z|^{d-2}},\quad |z|\to\infty.
\end{equation*}
If $d=4$ or $d=3$ then the same is valid provided that respectively
$\mathbf{E}|X|^2\log|X|<\infty$ or $\mathbf{E}|X|^2<\infty$.

Uchiyama mentions also that this moment condition is optimal:
for every $\varepsilon>0$ there exists a random walk satisfying
$\mathbf{E}|X|^{d-2-\varepsilon}<\infty$ such that 
$$
\limsup_{|z|\to\infty}|z|^{d-2}G_{\mathbb{R}^d}(0,z)=\infty.
$$
He has considered the dimensions $4$ and $5$ only, but it is quite simple
to show that this statement holds in every dimension $d\ge5$. 
We now give an example in our setting of a random walk which shows the optimality of 
Uchiyama's condition and of the moment condition in Theorem~\ref{thm:deep}.
Our example is just a multidimensional variation
of the classical Williamson example, see \cite{Williamson68}.
\begin{example}
\label{ex:Will-type}
Let $d$ be greater than $4$ and consider $X$ with the following distributon.
For every $n\ge1$ and for every basis vector $e_k$ put 
$$
\mathbf{P}(X=\pm 2^n e_k)=\frac{q_n}{2d},
$$
where the sequence $q_n$ is such that 
$$
\sum_{n=1}^\infty q_n=1\quad\text{and}\quad
q_n\sim \frac{c\log n}{2^{n(d-2)}}.
$$
Clearly,
$$
\mathbf{E}|X|^{d-2}=\infty\quad\text{and}\quad
\mathbf{E}\frac{|X|^{d-2}}{\log^{1+\varepsilon}|X|}<\infty.
$$
Using now the obvious inequality $G_{\mathbb{R}^d}(0,x)\ge\mathbf{P}(X=x)$,
we conclude that, for every $j=1,2,\ldots,d$,
$$
\lim_{n\to\infty}2^{(d-2)n}G_{\mathbb{R}^d}(0,\pm 2^n e_j)=\infty.
$$

If we have a cone $K$ such that $p\ge2$ and $e_j\in\Sigma$ for some $j$, then,
choosing $q_n\sim\frac{c\log n}{2^{n(p+d-2)}}$, we also have
$$
\lim_{n\to\infty}2^{(p+d-2)n}G_{K}(e_j,(1+2^n)e_j)=\infty.
$$
Therefore, the finiteness of $\e|X(1)|^{r_1(p)}$ can not be replaced by a weaker moment assumption.
\hfill$\diamond$
\end{example}
But Uchiyama shows that the moment assumption $\mathbf{E}|X|^{d-2}$ is not necessary, as it can be replaced by $\mathbf{P}(X=x)=o(|x|^{-d-2})$, which implies the existence of the second moment
only. In Theorem~\ref{thm:deep} we have a similar situation: the moment condition $\mathbf{E}|X|^{r_1(p)}<\infty$
is not necessary and can be replaced by the assumption \eqref{local_assump}, which yields the finiteness
of $\mathbf{E}|X|^{p\vee2}$ only. It has been shown in \cite{DW15}, if $p>2$ then the condition
$\mathbf{E}|X|^{p}<\infty$ is an optimal moment condition for the existence of the harmonic function
$V(x)$. 

We now turn to the asymptotic behavior of the Green function along the boundary of the cone.
\begin{theorem}
\label{thm:boundary}
Assume that $K$ is convex and $C^2$. Assume also that $\mathbf{E}|X|^{r_1(p)+1}<\infty$. If $y/|y|$
converges to $\sigma\in\partial\Sigma$ as $|y|\to\infty$ then there exists a strictly positive function
$v_\sigma(y)$ such that
\begin{equation}
\label{T3}
G_K(x,y)\sim c\frac{V(x)v_\sigma(dist(y,\partial K))}{|y|^{p+d-1}}.
\end{equation}
The function $v_\sigma$ is asymptotically linear, that is,
$$
v_\sigma(t)\sim c_\sigma t\quad\text{as }t\to\infty.
$$
Moreover, the same relation for $G_K$ holds if one replaces the moment condition\\
$\mathbf{E}|X|^{r_1(p)+1}<\infty$ by the following restriction on the local
structure of $X_1$:
\begin{equation}
\label{local_assumprand}
\mathbf{P}(X=x)\le |x|^{-p-d+1}f(|x|)
\end{equation}
for some decreasing function $f$ such that $\log(u)u^{(3-p)\vee 1}f(u)\to 0$ as
$u\to\infty$.
\end{theorem}
Clearly, one can adapt the random walk from Example~\ref{ex:Will-type} to show that the moment 
assumption in Theorem~\ref{thm:boundary} is minimal as well. Indeed, it suffices to take 
$q_n\sim\frac{c\log n}{2^{n(p+d-1)}}$ and to assume that one of the vectors $\pm e_j$ belongs to the 
boundary of the cone $K$.

Theorems~\ref{thm:deep} and \ref{thm:boundary} describe the asymptotic behavior of $G_k(x,y)$ along all 
possible directions inside the cone $K$. Combining these two results, we conclude that, for all $x,x'\in K$,
\[
\frac{G(x,y)}{G(x',y)}\to\frac{V(x)}{V(x')}\quad\text{as }|y|\to\infty.
\]
As a result we have the following
\begin{corollary}
\label{thm:martin}
Assume that the assumptions of Theorem~\ref{thm:boundary} are valid. Then the function $V(x)$ is the 
unique, up to mutiplicative factor, positive harmonic function for $\{S(n)\}$ killed at leaving $K$.
\end{corollary}
Doney~\cite{Doney98} has shown that the harmonic function for any one-dimensional oscillating random
walk killed at leaving the positive half-axis is unique without any additional moment assumption.

For multidimensional cones much less is known. Raschel~\cite{Raschel14} has shown the uniqueness of the 
positive harmonic function for random walks with small steps killed at leaving the positive quadrant
$\mathbb{Z}_+^2$ and some particular cases of such random walks have been studied by the same author in
\cite{Raschel10} and \cite{Raschel11}. The approach in \cite{Raschel14} is based on a functional
equation which is satisfied by all harmonic functions. It should be also mentioned that \cite{Raschel10}
and \cite{Raschel11} describe actually the asymptotic behavior of the Green function and the uniqueness
of the harmonic function is just a consequence of the results on the Green function.

Uchiyama~\cite{Uchiyama14} derives asymptotics for the Green function for random walks in 
$\mathbb{Z}^{d-1}\times\mathbb{Z}_+$, see Theorem 5 in \cite{Uchiyama14}. He assumes that $\mathbf{E}|X|^d$ is finite.
This is slightly weaker than the moment assumption in Theorem~\ref{thm:boundary}, which reduces in the
case of a half-space to $\mathbf{E}|X|^{d+1}<\infty$. On the other hand, Theorem~\ref{thm:boundary} can
be applied to any half-space, and Uchiyama's proof uses in a crucial way that the half-space is precisely
$\mathbb{Z}^{d-1}\times\mathbb{Z}_+$, i.e. aligned with the lattice of the random walk.

Bouaziz, Mustapha and Sifi~\cite{BMS2015} have shown uniqueness for a wide class of random walks
with finite number of steps killed at leaving the orthant $\mathbb{Z}_+^d$, $d\geq2$.

Ignatiouk-Robert~\cite{Ignatiouk18} has studied the properties of harmonic function for random walks on
semigroups by introducing a special renewal structure. Applying this to a random walk in a cone, she
has shown that, if the cone $K$ satisfies the assumptions from \cite{DW15} with some $p\le 2$ and
$\mathbf{E}|X|^\alpha$ is finite for some $\alpha>2$, then the harmonic function is unique. From our results
on the Green function one can deduce uniqueness if it is assumed that
$\mathbf{E}|X|^{d+1}<\infty$. Therefore, for cones with $p\le 2$ the result in \cite{Ignatiouk18} holds
under much weaker restrictions on the random walk than Corollary~\ref{thm:martin}. This does not
provide any information on the asymptotic behavior of the Green function.

Raschel and Tarrago \cite{Raschel_Tarrago} have derived an implicit asymptotic representation for the
Green function of random walks in cones under the assumptions $p\ge1$ and $\mathbf{E}|X|^{r_1(p)+p}<\infty$.
This is the main difference between their assumptions and ours: we need less moments but we 
impose the condition that the cone is $C^2$. The latter excludes, for example, Weyl chambers,
which appear often in models from physics.

We conclude this section by describing our approach.
Using the local limit theorem for $S(n)$ killed at leaving $K$ one can easily determine the asymptotic
behaviour of the sum
$$
\sum_{n\ge\varepsilon|y|^2}\mathbf{P}(x+S(n)=y,\tau_x>n),
$$
for every fixed $\varepsilon>0$. Thus, the main problem in studying $G_K(x,y)$ consists in estimating
the sum of local probabilities over $n\le \varepsilon |y|^2$. The local limit theorem is useless in
this domain, since $y$ belongs to the region of large deviations. For unconditioned one-dimensional
transient random walks Caravenna and Doney~\cite{Caravenna-Doney16} estimate every probability
separately to obtain asymptotics for the Green function. Our strategy is completely different:
we consider the sum of probabilities as the expectation of the number of visits to $y$ up to time 
$\varepsilon|y|^2$. In the proof of Theorem~\ref{thm:deep} we derive an upper bound for this
expectation in terms of the Green function $G_{\mathbb{R}^d}$ and apply the functional limit theorems
from \cite{DuW15}. It turns out that if $y$ goes to infinity along the boundary then the estimate
via $G_{\mathbb{R}^d}$ becomes too rough. For that reason we first consider the special case when $K$
is a half-space. In this situation we derive the asymptotic for $G_K$ via appropriate estimates
for local probabilities of large deviations. This part is similar to the approach in
\cite{Caravenna-Doney16}. Given the well-known asymptotic of Green functions for half-spaces we follow the same
strategy as in the proof of Theorem~\ref{thm:deep} to estimate the expected number of visits to $y$
in terms of the Green function for the half-space whose normal is perpendicular to the direction of
convergence for the Green function. For that estimate we need to assume that
$K$ is convex and $C^2$.

\section{Asymptotic behavior inside the cone.}
\begin{proof}[Proof of Theorem \ref{thm:deep}]
Fix some $\varepsilon>0$ and split $G_K(x,y)$ into two parts:
\begin{align*}
G_K(x,y)&=\sum_{n<\varepsilon|y|^2}\mathbf{P}(x+S(n)=y,\tau_x>n)
+\sum_{n\ge\varepsilon |y|^2}\mathbf{P}(x+S(n)=y,\tau_x>n)\\
&=:S_1(x,y,\varepsilon)+S_2(x,y,\varepsilon).
\end{align*}
By Theorem 5 in \cite{DW15}, 
$$
n^{p/2+d/2}\mathbf{P}(x+S(n)=y,\tau_x>n)=
\varkappa H_0 V(x)u\left(\frac{y}{\sqrt{n}}\right)e^{-|y|^2/2n}+o(1)
$$
uniformly in $y\in K$. Consequently, as $|y|\to\infty$, 
\begin{align*}
S_2(x,y,\varepsilon)
&=\varkappa H_0 V(x)\sum_{n\ge \varepsilon|y|^2}\frac{1}{n^{p/2+d/2}}
u\left(\frac{y}{\sqrt{n}}\right)e^{-|y|^2/2n}+
o\left(\sum_{n\ge \varepsilon|y|^2}\frac{1}{n^{p/2+d/2}}\right)\\
&=\varkappa H_0 V(x)u(y)\sum_{n\ge \varepsilon|y|^2}\frac{1}{n^{p+d/2}}e^{-|y|^2/2n}+
o\left(|y|^{-p-d+2}\right)\\
&=\varkappa H_0 V(x)u(y)|y|^{-2p-d+2}\int_\varepsilon^\infty z^{-p-d/2}e^{-1/(2z)}dz+
o\left(|y|^{-p-d+2}\right).
\end{align*}
Letting here $\varepsilon\to0$ and recalling that $u(y)\ge c(\alpha)|y|^p$ for $dist(y,\partial K)\ge \alpha|y|$,
we obtain
\begin{equation}
\label{lim_S_2}
\lim_{\varepsilon\to0}\lim_{|y|\to\infty}\frac{|y|^{2p+d-2}}{u(y)}S_2(x,y,\varepsilon)
=\varkappa H_0 V(x)\int_0^\infty z^{-p-d/2}e^{-1/(2z)}dz.
\end{equation}
Therefore, it remains to show that
\begin{align}
\label{lim_S_1}
\lim_{\varepsilon\to0}\limsup_{|y|\to\infty}\frac{|y|^{2p+d-2}}{u(y)}S_1(x,y,\varepsilon)=0.
\end{align}
Fix additionally some small $\delta>0$ and define
$$
\theta_y:=\inf\{n\ge1:x+S(n)\in B_{\delta,y}\},
$$
where $B_{\delta,y}$ denotes the ball of radius $\delta|y|$ around point $y$.

Then we have
\begin{align}
\label{S1_estim_1}
\nonumber
&S_1(x,y,\varepsilon)=\sum_{n<\varepsilon|y^2|}\mathbf{P}(x+S(n)=y,\tau_x>n\ge\theta_y)\\
\nonumber
&=\sum_{n<\varepsilon|y^2|}\sum_{k=1}^n\sum_{z\in B_{\delta,y}}
\mathbf{P}(x+S(k)=z,\tau_x>k=\theta_y)\mathbf{P}(z+S(n-k)=y,\tau_z>n-k)\\
\nonumber
&\le \sum_{k<\varepsilon|y|^2}\sum_{z\in B_{\delta,y}}
\mathbf{P}(x+S(k)=z,\tau_x>k=\theta_y)\sum_{j<\varepsilon|y|^2-k}\mathbf{P}(z+S(j)=y)\\
&\le\mathbf{E}\left[G^{(\varepsilon|y|^2)}(y-x-S(\theta_y));\tau_x>\theta_y,\theta_y\le\varepsilon|y|^2\right],
\end{align}
where
$$
G^{(t)}(z):=\sum_{n<t}\mathbf{P}(S(n)=z).
$$
We focus first on the case $d\ge3$. Then, according to Theorem 2 in Uchiyama~\cite{Uchiyama98},
\begin{equation}
\label{Green_full}
G(z):=G^{(\infty)}(z)\le\frac{C}{1+|z|^{d-2}},\quad z\in\mathbb{Z}^d,
\end{equation}
provided that $\mathbf{E}|X_1|^{s_d}<\infty$, where $s_d=2+\varepsilon$ for $d=3,4$ and
$s_d=d-2$ for $d\ge5$. Since $r_1(p)>s_d$, \eqref{Green_full} yields
\begin{align}
\label{S1_estim_2}
\nonumber
&S_1(x,y,\varepsilon)\\
&\le C\mathbf{E}\left[\frac{1}{1+|y-x-S(\theta_y)|^{d-2}};\tau_x>\theta_y,\theta_y\le\varepsilon|y|^2\right]\\
\nonumber
&\le C\mathbf{P}(|y-x-S(\theta_y)|\le\delta^2|y|,\tau_x>\theta_y,\theta_y\le\varepsilon|y|^2)
+\frac{C(\delta)}{|y|^{d-2}}\mathbf{P}(\tau_x>\theta_y,\theta_y\le\varepsilon|y|^2).
\end{align}
Noting now that $|y-x-S(\theta_y)|\le\delta^2|y|$ yields $|X(\theta_y)|>\delta(1-\delta)|y|$
and using our moment assumption, we conclude that
\begin{align}
\label{S1_estim_3}
\nonumber
&\mathbf{P}(|y-x-S(\theta_y)|\le\delta^2|y|,\tau_x>\theta_y,\theta_y<\varepsilon|y|^2)\\
\nonumber
&\hspace{1cm}\le \sum_{k<\varepsilon|y|^2}\mathbf{P}(|X(k)|>\delta(1-\delta)|y|,\tau_x>k=\theta_y)\\
\nonumber
&\hspace{1cm}\le \mathbf{P}(|X|>\delta(1-\delta)|y|)\sum_{k<\varepsilon|y|^2}\mathbf{P}(\tau_x>k-1)\\
&\hspace{1cm}=o\left(|y|^{-r_1(p)}\mathbf{E}[\tau_x;\tau_x<|y|^2]\right)=o\left(|y|^{-d-p+2}\right).
\end{align}
Recalling that $V$ is harmonic for $S(n)$ killed at leaving $K$, we obtain
\begin{align*}
&\mathbf{P}(\tau_x>\theta_y,\theta_y<\varepsilon|y|^2)\\
&\hspace{1cm}=\sum_{k<\varepsilon|y|^2}\sum_{z:|z-y|\le\delta|y|}\mathbf{P}(\tau_x>k,\theta_y=k,x+S(k)=z)\\
&\hspace{1cm}=\sum_{k<\varepsilon|y|^2}\sum_{z:|z-y|\le\delta|y|}\frac{V(x)}{V(z)}\mathbf{P}^{(V)}(\theta_y=k,x+S(k)=z)\\
&\hspace{1cm}\le\frac{V(x)}{\min_{\{z:|z-y|\le\delta|y|\}}V(z)}\mathbf{P}^{(V)}(\theta_y<\varepsilon|y|^2).
\end{align*}
It follows from the assumption $dist(y,\partial K)\ge\alpha|y|$ and Lemma 13 in \cite{DW15}, that
$$
\min_{\{z:|z-y|\le\delta|y|\}}V(z)\ge C|y|^p
$$
for all $\delta$ sufficiently small. As a result,
$$
|y|^p \mathbf{P}(\tau_x>\theta_y,\theta_y<\varepsilon|y|^2)
\le C(x)\mathbf{P}^{(V)}\left(\max_{n<\varepsilon|y|^2}|x+S(n)|>(1-\delta)|y|\right).
$$
Applying now the functional limit theorem for $S(n)$ under $\mathbf{P}^{(V)}$, see Theorem 2 and Corollary 3
in \cite{DuW15}, we conclude that
\begin{align}
\label{S1_estim_4}
\lim_{\varepsilon\to0}\limsup_{|y|\to\infty} |y|^p \mathbf{P}(\tau_x>\theta_y,\theta_y<\varepsilon|y|^2)=0.
\end{align}
Note that the functional limit theorem from \cite{DuW15} only requires $p\vee (2+\epsilon)$-moments.

Combining \eqref{S1_estim_2}--\eqref{S1_estim_4}, we infer that \eqref{lim_S_1} is valid under the 
assumption $\mathbf{E}|X_1|^{r_1(p)}<\infty$ in all dimensions $d\ge 3.$

Assume now that \eqref{local_assump} holds. It is clear that this restriction implies that
$\mathbf{E}|X_1|^p<\infty$. Therefore, Theorem 5 in \cite{DW15} is still applicable and \eqref{lim_S_2}
remains valid for all random walks satisfying \eqref{local_assump}. In order to show that \eqref{lim_S_1}
remains valid as well, we notice that
\begin{align*}
&S_1(x,y,\varepsilon)\\
&\hspace{0.5cm}\le C\mathbf{E}\left[\frac{1}{1+|y-x-S(\theta_y)|^{d-2}};|y-x-S(\theta_y)|\le\delta^2|y|,\tau_x>\theta_y,\theta_y\le\varepsilon|y|^2\right]\\
&\hspace{1.5cm}+\frac{C(\delta)}{|y|^{d-2}}\mathbf{P}(\tau_x>\theta_y,\theta_y\le\varepsilon|y|^2).
\end{align*}
In view of \eqref{S1_estim_4}, we have to estimate the first term on the right hand side only. 
For every $z$ such that $|z-y|\le \delta^2|y|$ we have
\begin{align*}
&\mathbf{P}(x+S(\theta_y)=z,\tau_x>\theta_y,\theta_y\le \varepsilon|y|^2)\\
&\hspace{1cm}\le\sum_{k=1}^{\varepsilon|y|^2}\sum_{z'\in K\setminus B_{\delta,y}}
\mathbf{P}(x+S(k-1)=z',\tau_x>k-1)\mathbf{P}(X_k=z-z').
\end{align*}
Since $|z-z'|>\delta(1-\delta)|y|$, we infer from \eqref{local_assump} that
\begin{align}
\label{local_1}
\nonumber
&\mathbf{P}(x+S(\theta_y)=z,\tau_x>\theta_y,\theta_y\le \varepsilon|y|^2)\\
\nonumber
&\hspace{1cm}\le C(\delta)|y|^{-p-d+1}f(\delta(1-\delta)|y|)\sum_{k=1}^{\varepsilon|y|^2}\mathbf{P}(\tau_x>k-1)\\
&\hspace{1cm} \le C(\delta)|y|^{-p-d+1}f(\delta(1-\delta)|y|)\mathbf{E}[\tau_x;\tau_x<|y|^2].
\end{align}

Here and in the following we use that $\mathbf{E}[\tau_x;\tau_x<|y|^2]\sim C|y|^{-p+2}$ if $p\le 2$. 

For every natural $m$ there are $O(m^{d-1})$ lattice points $z$ such that $|z-y|\in(m,m+1]$.
Then, using \eqref{local_1}, we obtain
\begin{align*}
&\mathbf{E}\left[\frac{1}{1+|y-x-S(\theta_y)|^{d-2}};|y-x-S(\theta_y)|\le\delta^2|y|,\tau_x>\theta_y,\theta_y\le\varepsilon|y|^2\right]\\
&\hspace{1cm}\le C(\delta)|y|^{-p-d+1}f(\delta(1-\delta)|y|)\mathbf{E}[\tau_x;\tau_x<|y|^2]\sum_{m=1}^{\delta^2|y|}\frac{m^{d-1}}{1+m^{d-2}}\\
&\hspace{1cm} \le C(\delta)|y|^{-p-d+3}f(\delta(1-\delta)|y|)\mathbf{E}[\tau_x;\tau_x<|y|^2].
\end{align*}
Recalling that $u^{(3-p)\vee1}f(u)\to0$, we conclude that
\begin{align*}
&\mathbf{E}\left[\frac{1}{1+|y-x-S(\theta_y)|^{d-2}};|y-x-S(\theta_y)|\le\delta^2|y|,\tau_x>\theta_y,\theta_y\le\varepsilon|y|^2\right]\\
&\hspace{2cm}=o(|y|^{-p-d+2}).
\end{align*}
This completes the proof of the theorem for $d\ge 3$.

We now focus on $d=2$.
In this case we can not use the full Green function.
We will obtain bounds for $G^{(t)}(x)$ directly from the local limit theorem for unrestricted walks.
More precisely, we shall use Propositions 9 and 10 from Chapter 2 in Spitzer's book \cite{Spitzer}, which say that
\begin{equation}
\label{eq_Spitzer}
\mathbf{P}(S_n=z)=\frac{1}{2\pi n}e^{-|z|^2/2n}+\frac{\rho(n,z)}{|z|^2\vee n},\quad \text{as } n\rightarrow \infty
\end{equation}
where
$$
\sup_{z\in\mathbb{Z}^2}\rho(n,z)\to0 ,\quad \text{as } n\rightarrow \infty.
$$
This asymptotic representation implies that
\begin{equation}
\label{Green_uniform}
\sup_{z\in\mathbb{Z}^2}G^{(t)}(z)\le C\log t,\quad t\ge2.
\end{equation}
Furthermore, for $|z|\to\infty$ and $t\le a|z|^2$ one has
\begin{equation*}
G^{(t)}(z)\le \sum_{n=1}^{a|z|^2}\frac{1}{2\pi n} e^{-|z|^2/2n}+o(1)
=\frac{1}{2\pi}\int_0^a\frac{1}{v}e^{-1/2v}dv+o(1).
\end{equation*}
As a result,
\begin{equation}
\label{Green_large}
\sup_{z\in\mathbb{Z}^2}G^{(a|z|^2)}(z)\le C(a)<\infty.
\end{equation}
Using \eqref{Green_uniform} and \eqref{Green_large}, we obtain
\begin{align*}
S_1(x,y,\varepsilon)
&\le C\log|y|\mathbf{P}(|y-x-S(\theta_y)|\le\delta^2|y|,\tau_x>\theta_y,\theta_y\le\varepsilon|y|^2)\\
&\hspace{2cm}+C(\varepsilon)\mathbf{P}(\tau_x>\theta_y,\theta_y\le\varepsilon|y|^2)
\end{align*}
According to \eqref{S1_estim_3},
\begin{align*}
\mathbf{P}(|y-x-S(\theta_y)|\le\delta^2|y|,\tau_x>\theta_y,\theta_y\le\varepsilon|y|^2)
&=o(|y|^{-r_1(p)}\mathbf{E}[\tau_x;\tau_x<|y|^2])\\
&=o(|y|^{-p}/\log|y|).
\end{align*}
Combining this with \eqref{S1_estim_4}, we conclude that \eqref{lim_S_1} holds for $d=2$.
\end{proof}
\section{Random walks in a half-space}
In this section we shall consider a particular cone
$$
K=\left\{x\in\mathbb{R}^d:x_d>0\right\}.
$$
Since the rotations of the space do not affect our moment assumptions, the results of this section
remain valid for any half-space in $\mathbb{R}^d$.

For this very particular cone we have
\begin{itemize}
 \item $u(x)=x_d$;
 \item $\tau_x=\inf\{n\ge1:x_d+S_d(n)\le0\}$;
 \item $V(x)$ depends on $x_d$ only and is proportional to the renewal function of ladder heights of 
       the random walk $\{S_d(n)\}$.
\end{itemize}
In other words, the exit problem from $K$ is actually a one-dimensional problem. This allows the use of exiting results for one-dimensional walks. As a result we obtain the asymptotic behavior 
of the Green function.
\begin{theorem}
\label{thm:half-space}
Assume that $\mathbf{E}|X|^{d+1}<\infty$. Assume also that $x=(0,\ldots,0,x_d)$ with $x_d=o(|y|)$.
Then 
$$
G_K(x,y)\sim c\frac{V(x)V'(y)}{|y|^d}.
$$
Here, $V'$ is the harmonic function for the killed reversed random walk $\{-S(n)\}$.
\end{theorem}

The proof of this result is based on the following simple generalization of known results for cones.
\begin{lemma}
\label{lem:half-space}
Assume that $\mathbf{E}|X|^{2+\delta}$ is finite. Then, uniformly in $x\in K$ with $x_d=o(\sqrt{n})$
\begin{itemize}
 \item [(a)] $\mathbf{P}(\tau_x>n)\sim \varkappa V(x)n^{-1/2}$;
 \item [(b)] $\{\frac{x+S([nt])}{\sqrt{n}}\},\,t\in[0,1]$ conditioned on $\{\tau_x>n\}$ converges
             weakly to the Brownian meander in $K$;
 \item [(c)] $\sup_{y\in K}\Big|n^{1/2+d/2}\mathbf{P}(x+S(n)=y;\tau_x>n)-cV(x)\frac{y_d}{\sqrt{n}}e^{-|y-x|^2/2n}\Big|\to0$.         
\end{itemize}
\end{lemma}
\begin{proof}
The first statement is the well-known result for one-dimensional random walks, see Corollary 3 in
Doney~\cite{D12}. The second and the third 
statements for fixed starting points $x$ have been proved in \cite{DuW15} and in \cite{DW15} respectively.
To consider the case of growing $x_d$ one has to make only one change: Lemma 24 from \cite{DW15} should
be replaced by the estimate
$$
\lim_{n\to\infty}\frac{1}{V(x)}
\mathbf{E}\left[|x+S(\nu_n)|;\tau_x>\nu_n,|x+S(\nu_n)|>\theta_n\sqrt{n},\nu_n\le n^{1-\varepsilon}\right]=0
$$
uniformly in $x_d\le \theta_n\sqrt{n}/2.$ If $x_d\ge n^{1/2-\varepsilon}$ then $\nu_n=0$ and the 
expectation equals zero. If $x_d\le n^{1/2-\varepsilon}$ then one repeats the proof of Lemma 24 in 
\cite{DW15} with $p$ replaced by $1$ and uses the part (a) of the lemma to obtain a uniform in $x_d$ estimate
for the sum $\sum_{j\le n^{1-\varepsilon}}\mathbf{P}(\tau_x>j-1)$. (In \cite{DW15}, the Markov
inequality has been used, since one does not have the statement (a) in general cones.)
\end{proof}
\begin{lemma}
\label{lem:half-space-2}
Uniformly in $y$ with $y_d=o(\sqrt{n})$,
$$
\mathbf{P}(x+S(n)=y,\tau_x>n)\sim c\frac{V(x)V'(y)}{n^{1+d/2}}e^{-|y|^2/2n}.
$$
\end{lemma}
\begin{proof}
Set $m = [\frac{n}{2}]$ and write 
\begin{align*}
\pr(x+&S(n) = y,\tau_x>n) \\
&= \sum_{z\in K} \pr(x+S(n-m) = z,\tau_x>n-m)\pr(z+S(m) = y,\tau_z>m)\\
&=\sum_{z\in K} \pr(x+S(n-m) = z,\tau_x>n-m)\pr(y+S'(m) = z,\tau'_y>m),
\end{align*}
where
$$
S'(k)=-S(k),\quad k\ge0
$$
and
$$
\tau'_y:=\inf\{n\ge 1: y+S'(n)\notin K\}.
$$
Applying part (c) of Lemma \ref{lem:half-space} to the random walk $\{S'(n)\}$,
we obtain
\begin{align*}
\pr(x+&S(n) = y,\tau_x>n)\\
&=\frac{cV'(y)}{m^{1+d/2}}\sum_{z\in K}z_de^{-|z-y|^2/2m}\pr(x+S(n-m) = z,\tau_x>n-m)\\
&\hspace{1cm}+o\left(V'(y)m^{-1/2-d/2}\pr(\tau_x>n-m)\right).
\end{align*}
Using now Lemma~\ref{lem:half-space}(a), we get
\begin{align*}
\pr(x+&S(n) = y,\tau_x>n)\\
&=\frac{cV'(y)V(x)}{m^{1/2+d/2}(n-m)^{1/2}}
\e_x\left[\frac{S_d(n-m)}{\sqrt{m}}e^{-|S(n-m)-y|^2/2m}\Big|\tau_x>n-m\right]\\
&\hspace{1cm}+o\left(\frac{V(x)V'(y)}{m^{-1/2-d/2}(n-m)^{1/2}}\right).
\end{align*}
It follows from part (b) of the previous lemma that
\begin{align*}
&\e_x\left[\frac{S_d(n-m)}{\sqrt{m}}e^{-|S(n-m)-y|^2/2m}\Big|\tau_x>n-m\right]\\
&\hspace{1cm}\sim \e\left[\left(M_{K,d}(1)\right)e^{-|M_K-y/\sqrt{m}|^2/2}\right],
\end{align*}
where $M_K(t)=(M_{K,1}(t),M_{K,2}(t),\ldots,M_{K,d}(t))$ is the meander in $K$.

Since $K=\mathbb{R}^{d-1}\times\mathbb{R}_+$, all coordinates of $M_K$ are independent.
Furthermore, $M_{K,1}(t),\ldots,M_{K,d-1}(t)$ are Brownian motions and $M_{K,d}(t)$ is the 
one-dimensional Brownian meander. Combining these observations with $y_d=o(\sqrt{n})$, we
conclude that 
\begin{align*}
&\e\left[\left(M_{K,d}(1)\right)e^{-|M_K-y/\sqrt{m}|^2/2}\right]\\
&\hspace{1cm}
\sim \e\left[M_{K,d}(1)e^{-M_{K,d}^2/2}\right]\prod_{i=1}^{d-1}\e\left[e^{-(M_{K,i}(1)-y_i/\sqrt{m})^2/2}\right]\\
&\hspace{2cm}= C\prod_{i=1}^{d-1} e^{-y_i^2/4m}\sim Ce^{-|y|^2/2n}.
\end{align*}
This completes the proof.
\end{proof}

\begin{proof}[Proof of Theorem~\ref{thm:half-space}]
If $y$ is such that $y_d\ge\alpha |y|$ for some $\alpha>0$ then it suffices to repeat the proof of 
Theorem~\ref{thm:deep}. 

We consider then the 'boundary case' $y_d=o(|y|)$. 

Using Lemma~\ref{lem:half-space-2}, one obtains
easily
\begin{equation*}
\lim_{\varepsilon\to0}\lim_{|y|\to\infty}\frac{|y|^d}{V(x)v'(y)}S_2(x,y,\varepsilon)=c.
\end{equation*}
Namely, it follows that 
\begin{align*}
\lim_{\varepsilon\to0}&\lim_{|y|\to\infty}\frac{|y|^d}{V(x)v'(y)}S_2(x,y,\varepsilon) = c\lim_{\varepsilon\to0}\lim_{|y|\to\infty}\sum_{n\ge \epsilon|y|^2}|y|^dn^{-1-\frac{d}{2}}e^{-\frac{|y|^2}{2n}}\\&
= c\int_0^{\infty}\left(\frac{1}{v}\right)^{1+\frac{d}{2}}e^{-\frac{1}{2v}}dv
\end{align*}
and the last integral is finite. 
It follows that the theorem will be proven if we show that
\begin{equation}
\label{half-space.1}
\lim_{\varepsilon\to0}\lim_{|y|\to\infty}\frac{|y|^d}{V(x)V'(y)}S_1(x,y,\varepsilon)=0.
\end{equation}
Using an appropriate rotation we can reduce everything to the case $y_k=o(|y|)$ for every $k=2,\ldots, d-1$
and $y_1\sim |y|$. This also implies $y_d = o(|y|)$.

We first split the probability $\mathbf{P}(x+S(n)=y,\tau_x>n)$ into two parts:
\begin{align*}
\mathbf{P}(x+S(n)=y,\tau_x>n)=&\mathbf{P}(x+S(n)=y,\tau_x>n,\max_{k\le n}|X_1(k)|\le\gamma y_1)\\
&+\mathbf{P}(x+S(n)=y,\tau_x>n,\max_{k\le n}|X_1(k)|>\gamma y_1),
\end{align*}
where $\gamma\in(0,1)$. Introduce the stopping time
$$
\sigma_\gamma:=\inf\{k\ge1:|X_1(k)|>\gamma y_1\}.
$$
Then, by the Markov property,
\begin{align*}
&\mathbf{P}(x+S(n)=y,\tau_x>n,\max_{k\le n}|X_1(k)|>\gamma y_1)\\
&\hspace{1cm}=\sum_{k=1}^n \mathbf{P}(x+S(n)=y,\tau_x>n,\sigma_\gamma=k)\\
&\hspace{1cm}\le\sum_{k=1}^n\mathbf{P}(\tau_x>k-1)\mathbf{P}(|X_1|>\gamma y_1)\max_z\mathbf{P}(S(n-k)=z).
\end{align*}
Using now the bounds $\mathbf{P}(\tau_x>k)\le CV(x)k^{-1/2}$ and $\max_z\mathbf{P}(S(k)=z)\le Ck^{-d/2}$,
we obtain
\begin{align*}
&\mathbf{P}(x+S(n)=y,\tau_x>n,\max_{k\le n}|X_1(k)|>\gamma y_1)\\
&\hspace{1cm}\le CV(x)\mathbf{P}(|X_1|>\gamma y_1)\sum_{k=1}^n\frac{1}{\sqrt{k}}\frac{1}{(n-k+1)^{d/2}}\\
&\hspace{1cm}\le C V(x)\mathbf{P}(|X_1|>\gamma y_1)\frac{(\log n)^{{\rm 1}\{d=2\}}}{\sqrt{n}}.
\end{align*}

Here, in the last step we have splited the sum $\sum_{k=1}^n\frac{1}{\sqrt{k}}\frac{1}{(n-k+1)^{d/2}}$ into $\sum_{k=1}^\frac{n}{2}$ and $\sum_{k=\frac{n}{2}}^n$ and used elementary inequalities.

This implies that
\begin{align*}
&\sum_{n=1}^{\varepsilon|y|^2}\mathbf{P}(x+S(n)=y,\tau_x>n,\max_{k\le n}|X_1(k)|>\gamma y_1)\\
&\hspace{1cm}\le C\sqrt{\varepsilon}V(x)\mathbf{P}(|X_1|>\gamma y_1)|y|\left(\log |y|\right)^{{\rm 1}\{d=2\}}.
\end{align*}
As a result, for all random walks satisfying
$$
\mathbf{E}\left[|X|^{d+1}\left(\log |X|\right)^{{\rm 1}\{d=2\}}\right]<\infty,
$$
we have
\begin{equation}
\label{half-space.2} 
\sum_{n=1}^{\varepsilon|y|^2}\mathbf{P}(x+S(n)=y,\tau_x>n,\max_{k\le n}|X_1(k)|>\gamma y_1)
=o\left(\frac{V(x)}{|y|^d}\right).
\end{equation}

In order to estimate $\mathbf{P}(x+S(n)=y,\tau_x>n,\max_{k\le n}|X_1(k)|\le\gamma y_1)$ 
we shall perform the following change of measure:
$$
\overline{\mathbf{P}}(X(k)\in dz)
=\frac{e^{hz_1}}{\varphi(h)}\mathbf{P}(X(k)\in dz; |X_1(k)|\le\gamma y_1),
$$
where
$$
\varphi(h)=\mathbf{E}\left[e^{hX_1};|X_1|\le \gamma y_1\right].
$$
Therefore,
\begin{align}
\label{half-space.3}
\nonumber
&\mathbf{P}(x+S(n)=y,\tau_x>n,\max_{k\le n}|X_1(k)|\le\gamma y_1)\\
&\hspace{1cm}=e^{-hy_1}\varphi^n(h)\overline{\mathbf{P}}(x+S(n)=y,\tau_x>n).
\end{align}
According to (21) in Fuk and Nagaev \cite{FN71},
\begin{align*}
&e^{-hy_1}\varphi^n(h)\\
&\le
\exp\left\{-hy_1+hn\mathbf{E}[X_1;|X_1|\le\gamma y_1]+\frac{e^{h\gamma y_1}-1-h\gamma y_1}{\gamma^2y_1^2}
n\mathbf{E}[X_1^2;|X_1|\le\gamma y_1]\right\}.
\end{align*}
Choosing
\begin{equation}
\label{half-space.4}
h=\frac{1}{\gamma y_1}\log\left(1+\frac{\gamma y_1^2}{n\mathbf{E}[X_1^2;|X_1|\le\gamma y_1]}\right)
\end{equation}
and noting that
$$
\left|\mathbf{E}[X_1;|X_1|\le\gamma y_1]\right|
=\left|\mathbf{E}[X_1;|X_1|>\gamma y_1]\right|
\le\frac{1}{\gamma y_1}\mathbf{E}[X_1^2]=\frac{1}{\gamma y_1},
$$
we conclude that uniformly for $n\le \gamma|y|^2$ it holds
$$
e^{-hy_1}\varphi^n(h)\le \left(\frac{en}{\gamma y_1^2}\right)^{1/\gamma}.
$$
Plugging this into \eqref{half-space.3}, we obtain uniformly for $n\le \gamma|y|^2$
\begin{align}
\label{eq:half-space.5}
\nonumber
&\mathbf{P}\left(x+S(n)=y,\tau_x>n,\max_{k\le n}|X_1(k)|\le\gamma y_1\right)\\
&\hspace{1cm}\le C(\gamma)\left(\frac{n}{|y|^2}\right)^{1/\gamma}\overline{\mathbf{P}}(x+S(n)=y,\tau_x>n).
\end{align}
According to Theorem 6.2 in Esseen~\cite{Esseen68}, there exists an absolute constant $C$ such that
$$
\sup_z\overline{\mathbf{P}}(S(n)=z)\le\frac{C}{n^{d/2}}\chi^{-d/2},
$$
where
$$
\chi:=\sup_{u\ge 1}\frac{1}{u^2}\inf_{|t|=1}\overline{\mathbf{E}}\left[(t,X(1)-X(2));|X(1)-X(2)|\le u\right].
$$
Since $h$ defined in \eqref{half-space.4} converges to zero as $|y|\to\infty$ uniformly in $n\le \gamma|y|^2$,
$$
\overline{\mathbf{E}}\left[(t,X(1)-X(2));|X(1)-X(2)|\le u\right]
\to\mathbf{E}\left[(t,X(1)-X(2));|X(1)-X(2)|\le u\right]
$$
for every fixed $u$. Since $S(n)$ is truly $d$-dimensional under the original measure,
$\inf_{|t|=1}\mathbf{E}\left[(t,X(1)-X(2));|X(1)-X(2)|\le u\right]>0$ for all large values $u$.
As a result, there exists $\chi_0>0$ such that $\chi\ge\chi_0$ for all $|y|$ large enough and
all $n\le \gamma|y|^2$. Consequently,
\begin{equation}
\label{half-space.6}
\sup_z\overline{\mathbf{P}}(S(n)=z)\le \frac{C\chi_0^{-d/2}}{n^{d/2}}.
\end{equation}
Combining this bound with \eqref{eq:half-space.5}, we obtain for all $r\in(0,1)$, $\gamma<2/d$
\begin{align*}
&\sum_{n=1}^{|y|^{2-r}} \mathbf{P}\left(x+S(n)=y,\tau_x>n,\max_{k\le n}|X_1(k)|\le\gamma y_1\right)\\
&\hspace{0.5cm}\le C(\gamma)\chi_0^{-d/2}|y|^{-2/\gamma} \sum_{n=1}^{|y|^{2-r}}n^{1/\gamma-d/2}
\le C(\gamma)\chi_0^{-d/2}|y|^{-2/\gamma}|y|^{(2-r)(1/\gamma-d/2+1)},
\end{align*}
for all $n\le \gamma|y|^2$.   
If we choose $\gamma$ so small that $r(1/\gamma-d/2+1)>2$, then
\begin{equation}
\label{half-space.7}
\sum_{n=1}^{|y|^{2-r}} \mathbf{P}\left(x+S(n)=y,\tau_x>n,\max_{k\le n}|X_1(k)|\le\gamma y_1\right)
=o\left(\frac{1}{|y|^d}\right).
\end{equation}

In the case $n\ge|y|^{2-r}$ we can not ignore the condition $\tau_x>n$. By the Markov property at times
$n/3$ and $2n/3$ and by \eqref{half-space.6},
\begin{align*}
&\overline{\mathbf{P}}(x+S(n)=y,\tau_x>n)\\
&\le\sum_{z,z'}\overline{\mathbf{P}}(x+S(n/3)=z,\tau_x>n/3)\overline{\mathbf{P}}(z+S(n/3)=z')
\overline{\mathbf{P}}(z'+S(n/3)=y,\tau_{z'}>n/3)\\
&=\sum_{z,z'}\overline{\mathbf{P}}(x+S(n/3)=z,\tau_x>n/3)\overline{\mathbf{P}}(z+S(n/3)=z')
\overline{\mathbf{P}}(y+S'(n/3)=z',\tau_{y}>n/3)\\
&\le\frac{C}{n^{d/2}}\overline{\mathbf{P}}(\tau_x>n/3)\overline{\mathbf{P}}(\tau'_y>n/3).
\end{align*}
Therefore, it remains to show that, uniformly in $n\in[|y|^{2-r},|y|^2]$,
\begin{equation}
\label{half-space.8}
\overline{\mathbf{P}}(\tau_x>n/3)\le C\frac{1+x_d}{\sqrt{n}}.
\end{equation}
Indeed, from this estimate and from the corresponding estimate for the inversed walk we get
$$
\overline{\mathbf{P}}(x+S(n)=y,\tau_x>n)\le C\frac{(x_d+1)(y_d+1)}{n^{d/2+1}}.
$$
This implies with help of \eqref{eq:half-space.5} that
\begin{align*}
&\sum_{n=|y|^{2-r}}^{\varepsilon|y|^2}\mathbf{P}\left(x+S(n)=y,\tau_x>n,\max_{k\le n}|X_1(k)|\le\gamma y_1\right)\\
&\hspace{1cm}\le C\varepsilon^{1/\gamma-d/2}(x_d+1)(y_d+1)|y|^{-d}.
\end{align*}
Combining this with \eqref{half-space.2} and with \eqref{half-space.7}, we obtain \eqref{half-space.1}.

To derive \eqref{half-space.8} we first estimate some moments of the random walk $S_d(n)$ under 
$\overline{\mathbf{P}}$. By the definition of this probability measure,
$$
\overline{\mathbf{E}}[X_d]=\frac{1}{\varphi(h)}\mathbf{E}\left[X_de^{hX_1};|X_1|\le \gamma y_1\right].
$$
For the expectation on the right hand side we have the representation
\begin{align*}
&\mathbf{E}\left[X_de^{hX_1};|X_1|\le \gamma y_1\right]\\
&\hspace{1cm}=\mathbf{E}\left[X_d;|X_1|\le \gamma y_1\right]+h\mathbf{E}\left[X_dX_1;|X_1|\le \gamma y_1\right]\\
&\hspace{3cm}+\mathbf{E}\left[X_d(e^{hX_1}-1-hX_1);|X_1|\le \gamma y_1\right]\\
&\hspace{1cm}=-\mathbf{E}\left[X_d;|X_1|> \gamma y_1\right]-h\mathbf{E}\left[X_dX_1;|X_1|> \gamma y_1\right]\\
&\hspace{3cm}+\mathbf{E}\left[X_d(e^{hX_1}-1-hX_1);|X_1|\le \gamma y_1\right].
\end{align*}
In the last step we have used the equalities $\mathbf{E}[X_d]=\mathbf{E}[X_dX_1]=0$.
If
\begin{equation}
\label{half-space.9}
\mathbf{E}|X|^{3+\delta}<\infty,
\end{equation}
then, by the Markov inequality,
$$
\mathbf{E}\left[X_d;|X_1|> \gamma y_1\right]+h\mathbf{E}\left[X_dX_1;|X_1|> \gamma y_1\right]
=o(y_1^{-2})=o(n^{-1}).
$$
Therefore,
$$
\mathbf{E}\left[X_de^{hX_1};|X_1|\le \gamma y_1\right]=o(n^{-1})+
\mathbf{E}\left[X_d(e^{hX_1}-1-hX_1);|X_1|\le \gamma y_1\right].
$$
It is obvious that $|e^x-1-x|\le \frac{x^2}{2}e^{|x|}$. Therefore,
\begin{align*}
&\left|\mathbf{E}\left[X_d(e^{hX_1}-1-hX_1);|X_1|\le \gamma y_1\right]\right|\\
&\hspace{1cm}\le \frac{h^2}{2}\mathbf{E}\left[|X_d|X_1^2e^{h|X_1|};|X_1|\le\gamma y_1\right]\\
&\hspace{1cm}\le \frac{e}{2}h^2\mathbf{E}|X_d|X_1^2+h^2e^{h\gamma y_1}\mathbf{E}\left[|X_d|X_1^2;|X_1|>\frac{1}{h}\right]\\
&\hspace{1cm}\le \frac{e}{2}h^2\mathbf{E}|X_d|X_1^2+h^{2+\delta}e^{h\gamma y_1}\mathbf{E}|X_d||X_1|^{2+\delta}\\
&\hspace{1cm}\le \frac{e}{2}h^2\mathbf{E}|X|^3 + h^{2+\delta}e^{h\gamma y_1}\mathbf{E}|X|^{3+\delta}.
\end{align*}
Here in the last step we have used H\"older's inequality. 
It is immediate from the definition of $h$ that $h^2\le cn^{-1}$. Furthermore, if $n\ge|y|^{2-r}$
with some $r<\frac{\delta}{2}$, then $h^{2+\delta}e^{h\gamma y_1}=o(n^{-1})$. From these estimates and from
assumption \eqref{half-space.9}, we obtain
\begin{equation}
\label{half-space.10}
\left|\mathbf{E}\left[X_de^{hX_1};|X_1|\le \gamma y_1\right]\right|\le\frac{c}{n}
\end{equation}
uniformly in $n\in[|y|^{2-r},|y|^2]$. 

By the same arguments,
\begin{align}
\label{half-space.11}
\nonumber
\varphi(h)&=\mathbf{E}\left[e^{hX_1};|X_1|\le\gamma y_1\right]\\
\nonumber
&=\mathbf{P}(|X_1|\le \gamma y_1)+h\mathbf{E}\left[X_1;|X_1|\le\gamma y_1\right]
+\mathbf{E}\left[e^{hX_1}-1-hX_1;|X_1|\le\gamma y_1\right]\\
\nonumber
&=1-\mathbf{P}(|X_1|> \gamma y_1)-h\mathbf{E}\left[X_1;|X_1|>\gamma y_1\right]
+\mathbf{E}\left[e^{hX_1}-1-hX_1;|X_1|\le\gamma y_1\right]\\
&=1+o(n^{-1}).
\end{align}
Combining this with \eqref{half-space.10}, we finally obtain
\begin{equation}
\label{half-space.12}
\left|\overline{\mathbf{E}}X_d\right|\le\frac{c_1}{n}.
\end{equation}
We now turn to the second and the third moments of $X_d$ under $\overline{\mathbf{P}}$.
Using \eqref{half-space.11} and the moment assumption we have
\begin{align*}
\overline{\mathbf{E}}X_d^2&=\frac{1}{\varphi(h)}\mathbf{E}[X_d^2e^{hX_1};|X_1|\le\gamma y_1]
=(1+o(1))\mathbf{E}[X_d^2e^{hX_1};|X_1|\le\gamma y_1]\\
&=\mathbf{E}[X_d^2;|X_1|\le\gamma y_1]+o(1)+O\left(\mathbf{E}\left[X_d^2(e^{hX_1}-1);|X_1|\le\gamma y_1\right]\right)\\
&=1+o(1)+O\left(he^{h\gamma y_1}\right).
\end{align*}
Noting that $he^{h\gamma y_1}=o(1)$ for all $n\ge|y|^{2-r}$ we get
\begin{equation}
\label{half-space.13}
\overline{\mathbf{E}}X_d^2=1+o(1). 
\end{equation}
Similarly,
\begin{align*}
\overline{\mathbf{E}}|X_d|^3&=(1+o(1))\mathbf{E}[|X_d|^3e^{hX_1};|X_1|\le\gamma y_1]\\
&\le c\left(\mathbf{E}[|X_d|^3;|X_1|\le 1/h]+e^{h\gamma y_1}\mathbf{E}[|X_d|^3;|X_1|>1/h]\right)\\
&\le c\left(\mathbf{E}|X_d|^3+h^\delta e^{h\gamma y_1}\mathbf{E}|X_d|^{3+\gamma}\right)
\end{align*}
Using once again the fact that $h^\delta e^{h\gamma y_1}=o(1)$ for $n\ge|y|^{2-r}$, we arrive at
\begin{equation}
\label{half-space.14}
\overline{\mathbf{E}}|X_d|^3\le c_3.
\end{equation}

Now we can derive \eqref{half-space.8}. First, it follows from \eqref{half-space.12} that
$$
\overline{\mathbf{P}}(\tau_x>n/3)\le \overline{\mathbf{P}}(\tau^0_{x+c_1}>n/3),
$$
where
$$
\tau_y^0:=\inf\{k\ge1:y+S_d^0(k)\le0\}\quad\text{and}\quad
S_d^0(k)=S_d(k)-k\overline{\mathbf{E}}X_d.
$$
Applying Lemma 25 in \cite{DSW18} to the random walk $S_d^0$, we have
$$
\overline{\mathbf{P}}(\tau^0_y>k)
\le\frac{\overline{\mathbf{E}}[y+S_d^0(k);\tau_y^0>k]}{\overline{\mathbf{E}}(y+S_d^0(k))^+}
$$
Relations \eqref{half-space.13} and \eqref{half-space.14} allow the application of the central limit theorem
to the walk $S_d^0(k)$, which gives $\overline{\mathbf{E}}(y+S_d^0(k))^+\ge c\sqrt{k}$. Consequently,
$$
\overline{\mathbf{P}}(\tau^0_y>k)\le\frac{C}{\sqrt{k}}\overline{\mathbf{E}}[y+S_d^0(k);\tau_y^0>k].
$$
Further, by the optional stopping theorem,
\begin{align*}
\overline{\mathbf{E}}[y+S_d^0(k);\tau_y^0>k]&=y-\overline{\mathbf{E}}[y+S_d^0(\tau_y^0);\tau_y^0\le k]\\
&\le y-\overline{\mathbf{E}}[y+S_d^0(\tau_y^0)].
\end{align*}
We now use inequality (7) in \cite{M73} which states that there exists an absolute constant $A$ such that
$$
-\overline{\mathbf{E}}[y+S_d^0(\tau_y^0)]
\le A\frac{\overline{\mathbf{E}}|X_d|^3}{\overline{\mathbf{E}}X_d^2}.
$$
Combining this with \eqref{half-space.13} and \eqref{half-space.14}, we finally get
$$
\overline{\mathbf{P}}(\tau^0_y>k)\le \frac{C(y+1)}{\sqrt{k}},
$$
which implies \eqref{half-space.8}.
\end{proof}
\section{Asymptotics close to the boundary}
\subsection{Limit theorems for random walks starting far from the origin but close to the boundary}
Let $|y|\to\infty$ in such a way that $dist(y,\partial K)=o(|y|)$. 
Let $y_\perp\in\partial K$ be defined by the relation  $dist(y,\partial K)=|y-y_\perp|$.
Set $\sigma(y):=y_\perp/|y|\in\partial\Sigma$ and assume that $\sigma(\cdot)$ converges as $|y|\to\infty$ to some $\bar\sigma\in \partial\Sigma$.
Let $H_y$ denote a tangent hyperplane at point $y_\perp.$ Let $P_n$ be the distribution
of the linear interpolation of $t\ra (y+S(nt))/\sqrt{n}$ conditioned to stay in the half-space $K_y$ containing the cone K and having
boundary $H_y$. Then $P_n\to P$ weakly on $C[0,1]$.
Denote
$$
A_n:=\{f\in C[0,1]: f(k/n)\in K\text{ for all }1\le k\le n\}.
$$
Then 
$$
\liminf A_n\supseteq\{f\in C[0,1]: f(t)\in K\text{ for all }t\in(0,1]\}
$$
and
$$
\limsup \overline{A_n}\subseteq\{f\in C[0,1]: f(t)\in \overline{K}\text{ for all }t\in(0,1]\},
$$
where $\overline{A}$ denote the closure of $A$.

Denote for every fixed $n$ by $[0,1]\ni t\mapsto S(nt)$ the linear interpolation of $S(k), k\le n$. The conditions to apply Theorem 2.3 from Durrett \cite{Durrett78} are given. This leads to an invariance principle: $[0,1]\ni r\ra\frac{y+S(nr)}{\sqrt{n}}$ converges weakly as $\frac{n}{|y|^2}\to t$ to the Brownian meander $\{B_r,r\le 1\}$ inside the cone $K$ started at $\frac{\sigma}{\sqrt{t}}$. In particular it holds with $T_y:=\inf\{n\ge1: y+S(n)\notin K_y\}$
\begin{equation}
\label{1.step.1}
\mathbf{P}\left(\frac{y+S(n)}{\sqrt{n}}\in B\Big|\tau_y>n\right)\sim Q_{\sigma,t}(B)=\int_B q_{\sigma,t}(z)dz,\quad
\frac{n}{|y|^2}\to t
\end{equation}
where $q_{\sigma,t}(z)$ is the density of the Brownian meander in $K$, started at $\frac{\sigma}{\sqrt{t}}$ and evaluated at time $1$. 
Theorem 2.3 in \cite{Durrett78} also leads to 
\begin{equation}
\label{1.step.2}
\mathbf{P}(\tau_y>n|T_y>n)\to c_{\sigma,t}.
\end{equation}
where
$$
T_y:=\inf\{n\ge1: y+S(n)\notin K_y\}.
$$
Limiting relations \eqref{1.step.1} and \eqref{1.step.2} imply that
\begin{equation}
\label{V_low_bound}
V(y)\ge c|y|^{p-1}(1+dist(y,\partial K)).
\end{equation}
Indeed, by the harmonicity of $V$
$$
V(y)=\mathbf{E}[V(y+S(n));\tau_y>n],\ n\ge1.
$$
Fix now some $\epsilon>0$ and note that choosing $n=[|y|^2]$
it follows that $V(z)\sim u(z)$ uniformly as $z\to\infty$ as long as distance to $\partial K$ of $z$ is at least $\epsilon|z|$ (Lemma 13 in \cite{DW15}).
We obtain, as $|y|\to\infty$,
 and $\epsilon\rightarrow 0$
$$V(y)\ge \mathbf{P}(T_y>[|y|]^2)c_{\sigma,1}|y|^p\int_K u(z)q_{\sigma,1}(z)dz.
$$
Due to results for the one-dimensional random walk we arrive at
$$
\mathbf{P}(T_y>[|y|]^2)\ge c\frac{1+dist(y,\partial K)}{|y|}.
$$
This establishes \eqref{V_low_bound}.

Before proving Theorem \ref{thm:boundary} we record an auxiliary estimate needed in its proof.

\begin{lemma}\label{thm:auxmeander}
Define
$$
\phi_\sigma(t)=c_{\sigma,t}\int_K u(z)e^{-\frac{|z|^2}{2}}q_{\sigma,t}(z)dz.
$$
It holds $\phi_\sigma(t)=o(e^{-c/t})$ as $t\to 0$ for some $c>0$.
\end{lemma}

\begin{proof}
First we record that due to the invariance principle for the halfspace it holds 
\[
c_{\sigma,t} = \pr_{\sigma}(\tau^{me}>t) = \pr_{\sigma/\sqrt{t}}(\tau^{me}>1),
\]
where $\tau^{me}: = \inf\{t>0:M^{\sigma}(t)\not\in K_y\}$. Here $M^{\sigma}(t)$ is a Brownian meander
in $K_{y}$ whereas we will denote the Brownian meander in $K$ by $M_K^{\sigma}(t)$.
Since $|\sigma| = 1$ and $K$ is contained in $K_y$ it is clear then that $c_{\sigma,t}\ra 1$ as $t\ra 0$.

It follows 
\begin{align*}
\phi_{\sigma}(t)\le C\e_{\sigma/\sqrt{t}}\left[u(M_K^{\sigma}(1))e^{-\frac{|M_K^{\sigma}(1)|^2}{2}}\right]\le C\e_{\sigma/\sqrt{t}}\left[u(M^{\sigma}(1))e^{-\frac{|M^{\sigma}(1)|^2}{2}}\right].
\end{align*}
The second inequality can be easily justified using the invariance principles for meanders in $K$ and
$K_y$ as well as the fact that $c_{\sigma,t}\ra 1$ is bounded away from zero. 

It follows 

\[
\phi_{\sigma}(t)\le C\e_{\sigma/\sqrt{t}}\left[e^{-\frac{|M^{\sigma}(1)|^2}{4}}\right].
\]

Due to rotation invariance of Brownian motion the expectation doesn't depend on $\sigma$
so that we can choose $\sigma = (1,0,\ldots,0)$ and $K_y = \R^{d-1}\times \R_+$. The first
$d-1$ coordinates become independent Brownian motions whereas the last one is a
$1$-dimensional Brownian meander (see \cite{DIM77} for its density). This finishes the proof.



\end{proof}

\subsection{Proof of Theorem \ref{thm:boundary}} 
To estimate the contribution coming from large values of $n$
one does not need the limit theorems from the previous paragraph,
quite rough estimates turn out to be sufficient.

Set $m=[n/2]$. Then, applying the Markov property at time $m$
and inverting the time in the second part of the path, we obtain
\begin{align*}
&\mathbf{P}(x+S(n)=y,\tau_x>n)\\
&\hspace{1cm}=\sum_{z\in K}\mathbf{P}(x+S(m)=z,\tau_x>m)\mathbf{P}(y+S'(n-m)=z,\tau_y'>n-m)\\
&\hspace{1cm}\le\max_{z\in K}\mathbf{P}(x+S(m)=z,\tau_x>m)\mathbf{P}(\tau_y'>n-m).
\end{align*}
By Theorem 5 in \cite{DW15},
$$
\max_{z\in K}\mathbf{P}(x+S(m)=z,\tau_x>m)\le C\frac{V(x)}{m^{p/2+d/2}}.
$$
Furthermore, due to results for the one-dimensional random walk (see, for example Lemma 3 in \cite{DW16})
\begin{equation}
\label{3.step.1}
\mathbf{P}(\tau_y'>n-m)\le \mathbf{P}(T'_y>n-m)\le C\frac{1+dist(y,\partial K)}{\sqrt{n-m}}.
\end{equation}
Combining these estimates, we obtain
$$
\mathbf{P}(x+S(n)=y)\le CV(x)(1+dist(y,\partial K))n^{-(p+d+1)/2}.
$$
Consequently, for $A\ge2$ and $|y|\ge1$,
\begin{align}
\label{3.step}
\nonumber
\sum_{n\ge A|y|^2}\mathbf{P}(x+S(n)=y)
&\le CV(x)(1+dist(y,\partial K))\sum_{n\ge A|y|^2}n^{-(p+d+1)/2}\\
&\le CV(x)A^{-(p+d-1)/2}\frac{1+dist(y,\partial K)}{|y|^{p+d-1}}.
\end{align}

We turn now to the 'middle' part: $n\in(\varepsilon|y|^2,A|y|^2)$.
Using again the Markov property at time $m=[n/2]$ and applying Theorem 5 in \cite{DW15},
we obtain
\begin{align*}
&\mathbf{P}(x+S(n)=y,\tau_x>n)\\
&=\sum_{z\in K}\mathbf{P}(x+S(m)=z,\tau_x>m)\mathbf{P}(y+S'(n-m)=z;\tau'_y>n-m)\\
&=\frac{\varkappa H_0V(x)}{m^{p/2+d/2}}
\sum_{z\in K}\left(u\left(\frac{z}{\sqrt{m}}\right)e^{-\frac{|z|^2}{2m}}+o(1)\right)\mathbf{P}(y+S'(n-m)=z;\tau'_y>n-m)\\
&=\frac{\varkappa H_0V(x)}{m^{p/2+d/2}}
\mathbf{E}\left[u\left(\frac{S'(n-m)}{\sqrt{m}}\right)e^{-\frac{|S'(n-m)|^2}{2m}};\tau'_y>n-m\right]
+o\left(\frac{\mathbf{P}(\tau'_y>n-m)}{m^{p/2+d/2}}\right).
\end{align*}
Taking into account \eqref{3.step.1}, we have
\begin{align*}
&\mathbf{P}(x+S(n)=y,\tau_x>n)\\
&=\frac{\varkappa H_0V(x)}{m^{p/2+d/2}}
\mathbf{E}\left[u\left(\frac{S'(n-m)}{\sqrt{m}}\right)e^{-\frac{|S'(n-m)|^2}{2m}};\tau'_y>n-m\right]
+o\left(\frac{1+dist(y, \partial K)}{n^{(p+d+1)/2}}\right). 
\end{align*}
Next, it follows from \eqref{1.step.1} and \eqref{1.step.2} that if $\frac{n}{|y^2|}\sim t$ then
\begin{align*}
\mathbf{E}\left[u\left(\frac{S'(n-m)}{\sqrt{m}}\right)e^{-\frac{|S'(n-m)|^2}{2m}};\tau'_y>n-m\right]
\sim \mathbf{P}(T'_y>n-m) \phi_\sigma(t/2).
\end{align*}
Since $T'_y$ is an exit time from a half space,
$$
\mathbf{P}(T'_y>k)\sim v'(y)k^{-1/2},
$$
where $v'(y)$ is the positive harmonic function for $S'$ killed at leaving the half-space $K_{\sigma}$.
As a result,
\begin{equation*}
\mathbf{P}(x+S(n)=y,\tau_x>n)
=C_0 \frac{V(x)v'(y)}{n^{(p+d+1)/2}}\phi_\sigma\left(\frac{n}{|y|^2}\right)
+o\left(\frac{1+dist(y, \partial K)}{n^{(p+d+1)/2}}\right).
\end{equation*}
where
$$
C_0:=\varkappa H_0 2^{(p+d+1)/2}.
$$
This representation implies that
\begin{align*}
&\sum_{\varepsilon|y|^2}^{A|y|^2}\mathbf{P}(x+S(n)=y,\tau_x>n)\\
&\hspace{1cm}= C_0V(x)v'(y)\sum_{\varepsilon|y|^2}^{A|y|^2}n^{-(p+d+1)/2}\phi_\sigma\left(\frac{n}{2|y|^2}\right)
+o\left(\frac{1+dist(y, \partial K)}{n^{(p+d-1)/2}}\right)\\
&\hspace{1cm}= C_0\frac{V(x)v'(y)}{|y|^{p+d-1}}\int_\varepsilon^A\phi_\sigma(t/2) t^{-(p+d+1)/2}dt
+o\left(\frac{1+dist(y, \partial K)}{n^{(p+d-1)/2}}\right).
\end{align*}
Combining this with \eqref{3.step} and letting $A\to\infty$, one can easily obtain
\begin{equation*}
\lim_{|y|\to\infty}\frac{|y|^{p+d-1}}{V(x)v'(y)}S_2(x,y,\varepsilon)
=C_0\int_\varepsilon^\infty \phi_\sigma(t/2) t^{-(p+d+1)/2}dt.
\end{equation*}
From Lemma \ref{thm:auxmeander} it follows
\begin{equation}
\label{S2.boundary}
\lim_{\varepsilon\to0}\lim_{|y|\to\infty}\frac{|y|^{p+d-1}}{V(x)v'(y)}S_2(x,y,\varepsilon)
=C_0\int_0^\infty \phi_\sigma(t/2) t^{-(p+d+1)/2}dt.
\end{equation}

It remains to estimate $S_1(x,y,\varepsilon)$. We shall use the same strategy as in the proof of 
Theorem~\ref{thm:deep}, but instead of the Green function for the whole space we shall use the Green
function for the half-space $K_y$. More precisely,
\begin{align*}
&S_1(x,y,\varepsilon)=\sum_{n<\varepsilon|y^2|}\mathbf{P}(x+S(n)=y,\tau_x>n\ge\theta_y)\\
&=\sum_{n<\varepsilon|y^2|}\sum_{k=1}^n\sum_{z\in B_{\delta,y}}
\mathbf{P}(x+S(n)=z,\tau_x>k=\theta_y)\mathbf{P}(z+S(n-k)=y,\tau_z>n-k)\\
&= \sum_{k<\varepsilon|y|^2}\sum_{z\in B_{\delta,y}}
\mathbf{P}(x+S(n)=z,\tau_x>k=\theta_y)\sum_{j<\varepsilon|y|^2-k}\mathbf{P}(z+S(j)=y,\tau_z>j)\\
&\le \sum_{k<\varepsilon|y|^2}\sum_{z\in B_{\delta,y}}
\mathbf{P}(x+S(n)=z,\tau_x>k=\theta_y)\sum_{j<\varepsilon|y|^2}\mathbf{P}(y+S'(j)=z,T_y'>j)\\
&=\mathbf{E}\left[G_{\varepsilon, y}(x+S(\theta_y));\tau_x>\theta_y,\theta_y\le\varepsilon|y|^2\right],
\end{align*}
where
$$
G_{\varepsilon, y}(z)=\sum_{j<\varepsilon|y|^2}\mathbf{P}(y+S'(j)=z,T_y'>j).
$$
Applying Theorem \ref{thm:half-space} and \eqref{Green_large} to the random walk $S'(n)$, we obtain
$$
G_{\varepsilon, y}(z)\le
C\frac{v'(y)(1+dist(z,H_y))}{1+|z-y|^d}\wedge1.
$$
Therefore,
\begin{align}
\label{S1_boundary.1}
\nonumber
S_1(x,y,\varepsilon)
&\le C\mathbf{P}(|y-x-S(\theta_y)|\le\delta^2|y|,\tau_x>\theta_y,\theta_y\le\varepsilon|y|^2)\\
&\hspace{0.5cm}+C(\delta)\frac{v'(y)}{|y|^d}\mathbf{E}\left[(1+dist(x+S(\theta_y),H_y);
\tau_x>\theta_y,\theta_y\le\varepsilon|y|^2\right].
\end{align}
The first term has been estimated in \eqref{S1_estim_3}:
\begin{align}
\label{S1_boundary.2}
\mathbf{P}(|y-x-S(\theta_y)|\le\delta^2|y|,\tau_x>\theta_y,\theta_y\le\varepsilon|y|^2)
=o(|y|^{-p-d+1})
\end{align}
for random walks having finite moments of order $r_2(p):=p+d-1+(2-p)^+$.

In order to estimate the second term in \eqref{S1_boundary.1}, we shall perform again the change
of measure with the harmonic function $V$:
\begin{align*}
&\mathbf{E}\left[(1+dist(x+S(\theta_y),H_y);\tau_x>\theta_y,\theta_y\le\varepsilon|y|^2\right]\\
&\hspace{0.5cm}
=V(x)\mathbf{E}^{(V)}\left[\frac{1+dist(x+S(\theta_y),H_y)}{V(x+S(\theta_y))};
\theta_y\le\varepsilon|y|^2\right].
\end{align*}
Applying now \eqref{V_low_bound}, we obtain
\begin{align*}
\mathbf{E}\left[(1+dist(x+S(\theta_y),H_y);\tau_x>\theta_y,\theta_y\le\varepsilon|y|^2\right]
\le CV(x)|y|^{-p+1}\mathbf{P}^{(V)}(\theta_y\le\varepsilon|y|^2).
\end{align*}
From this estimate and \eqref{S1_estim_4} we conclude that
\begin{equation*}
\lim_{\varepsilon\to0}\lim_{|y|\to\infty}|y|^{p-1}
\mathbf{E}\left[(1+dist(x+S(\theta_y),H_y);\tau_x>\theta_y,\theta_y\le\varepsilon|y|^2\right]=0.
\end{equation*}
Combining this estimate with \eqref{S1_boundary.1} and \eqref{S1_boundary.2} as well as Lemma 13 in \cite{DW15} we get
\begin{equation}
\label{S1.boundary}
\lim_{\varepsilon\to0}\lim_{|y|\to\infty}|y|^{p+d-1}S_1(x,y,\varepsilon)=0.
\end{equation}
Since $v'(y)$ is bounded from below by a positive number, \eqref{S1.boundary} and \eqref{S2.boundary}
yield the desired result for the case $\e[|X|^{r_2(p)}]<\infty$ due to classical results for the one-dimensional random walk.

Assume now that \eqref{local_assumprand} holds. It is easy to see that the above proof that 
\begin{equation}
\lim_{\varepsilon\to0}\lim_{|y|\to\infty}\frac{|y|^{p+d-1}}{V(x)v'(y)}S_2(x,y,\varepsilon)
=C_0\int_0^\infty \phi_\sigma(t) t^{-(p+d+1)/2}dt,
\end{equation}
goes through again word for word. Therefore we focus on the asymptotic of $S_1(x,y,\varepsilon)$ in the following. 
With similar steps as above it holds 
\begin{align*}
&S_1(x,y,\varepsilon)\\
&\le C(\delta)v'(y)\e\left[\frac{1+dist(x+S(\theta_y),H_y)}{1+|x+S(\theta_y)-y|^d},|y-x-S(\theta_y)|\le\delta^2|y|,\tau_x>\theta_y,\theta_y\le\varepsilon|y|^2\right]\\&+
C(\delta)\frac{v'(y)}{|y|^d}\mathbf{E}\left[(1+dist(x+S(\theta_y),H_y);
\tau_x>\theta_y,\theta_y\le\varepsilon|y|^2\right].
\end{align*}
The second summand can be treated just as above with help of \eqref{V_low_bound} so that we need to show
\begin{align*}
\e\left[\frac{1+dist(x+S(\theta_y),H_y)}{1+|x+S(\theta_y)-y|^d},|y-x-S(\theta_y)|\le\delta^2|y|,\tau_x>\theta_y,\theta_y\le\varepsilon|y|^2\right]\\
= O(|y|^{-p-d+1}).
\end{align*}
It holds
$$
1+dist(x+S(\theta_y),H_y)\le 1+|S(\theta_y)-y|+|y-y_{\perp}|= o(|y|)+ |S(\theta_y)-y|.
$$
To complete the proof we now show for $r = d-1,d$ 
\begin{align*}
\label{eq:helpass4}
&S_{2,r}(x,y,\varepsilon)\\
&\hspace{1cm}:=\e\left[\frac{1}{1+|x+S(\theta_y)-y|^{r}},|y-x-S(\theta_y)|\le\delta^2|y|,\tau_x>\theta_y,\theta_y\le\varepsilon|y|^2\right]\\
&\hspace{1cm} = o(|y|^{-p-d+1}).
\end{align*}

With a similar calculation as in the proof of Theorem \ref{thm:deep} (using \eqref{local_1}) we obtain
\begin{align*}
&\mathbf{E}\left[\frac{1}{1+|y-x-S(\theta_y)|^{d-1}};|y-x-S(\theta_y)|\le\delta^2|y|,\tau_x>\theta_y,\theta_y\le\varepsilon|y|^2\right]\\
&\hspace{1cm}\le C(\delta)|y|^{-p-d+1}f(\delta(1-\delta)|y|)\mathbf{E}[\tau_x;\tau_x<|y|^2]\sum_{m=1}^{\delta^2|y|}\frac{m^{d-1}}{m^{d-1}}\\
&\hspace{1cm} \le C(\delta)|y|^{-p-d+2}f(\delta(1-\delta)|y|)|y|^{(2-p)^+}.
\end{align*}

Finally, 

\begin{align*}
&\mathbf{E}\left[\frac{1}{1+|y-x-S(\theta_y)|^{d}};|y-x-S(\theta_y)|\le\delta^2|y|,\tau_x>\theta_y,\theta_y\le\varepsilon|y|^2\right]\\
&\hspace{1cm}\le C(\delta)|y|^{-p-d+1}f(\delta(1-\delta)|y|)\mathbf{E}[\tau_x;\tau_x<|y|^2]\sum_{m=1}^{\delta^2|y|}\frac{m^{d-1}}{m^{d}}\\
&\hspace{1cm} \le C(\delta)\log(|y|)|y|^{-p-d+2}f(\delta(1-\delta)|y|)|y|^{(2-p)^+}.
\end{align*}

This finishes the proof of Theorem \ref{thm:boundary}.

\end{document}